\documentclass[a4paper,12pt,final]{amsart}
\usepackage{times,a4wide,mathrsfs,amssymb,amsmath,amsthm}

\newcommand{\C}{\mathbb{C}}
\newcommand{\ZZ}{\mathbb{Z}}

\newcommand{\QQ}{\mathbb{Q}}
\newcommand{\NN}{\mathbb{N}}
\newcommand{\PP}{\mathbb{P}}
\newcommand{\LL}{\mathcal{L}}

\newcommand{\OO}{\mathcal O}

\newcommand{\pp}{\mathfrak p}

\newcommand{\DD}{\mathcal D}
\newcommand{\XX}{\mathcal X}
\newcommand{\YY}{\mathcal Y}

\newcommand{\CC}{\mathcal C}

\newcommand{\EE}{\mathcal E}
\newcommand{\MM}{\mathcal M}
\newcommand{\PPP}{\mathcal P}
\newcommand{\FF}{\mathcal F}

\newcommand{\wt}{\widetilde}
\newcommand{\rom}{\romannumeral}

\DeclareMathOperator{\aut}{Aut}
\DeclareMathOperator{\ide}{id}

\DeclareMathOperator{\ima}{Im}

\newtheorem{theorem}{Theorem}[section]

\newtheorem{lemma}[theorem]{Lemma}

\newtheorem{corollary}[theorem]{Corollary}
\newtheorem{proposition}[theorem]{Proposition}
\newtheorem{conjecture}[theorem]{Conjecture}
\newtheorem{remark}[theorem]{Remark}
\newtheorem{definition}[theorem]{Definition}
\newtheorem{convention}{Conventions}

\newtheorem{notation}[theorem]{Notation}

\newtheorem{nonumbering}{Theorem}

\newtheorem{nonumberingc}{Corollary}

\newtheorem{nonumberingt}{Acknowledgements}

\begin{document}
\author[Robert Laterveer]
{Robert Laterveer}

\address{Institut de Recherche Math\'ematique Avanc\'ee,
CNRS -- Universit\'e 
de Strasbourg,\
7 Rue Ren\'e Des\-car\-tes, 67084 Strasbourg CEDEX,
FRANCE.}
\email{robert.laterveer@math.unistra.fr}

\title[Algebraic cycles on certain HK fourfolds II]{Algebraic cycles on certain hyperk\"ahler fourfolds with an order $3$ non--symplectic automorphism II}

\begin{abstract} Let $X$ be a hyperk\"ahler variety, and assume that $X$ admits a non--symplectic automorphism $\sigma$ of order $k>{1\over 2}\dim X$. Bloch's conjecture predicts that the quotient
$X/<\sigma>$ should have trivial Chow group of $0$--cycles. We verify this for Fano varieties of lines on certain cubic fourfolds
having an order $3$ non--symplectic automorphism.
\end{abstract}


\maketitle

\section{Introduction}

Let $X$ be a smooth projective variety over $\C$, and let $A^i(X):=CH^i(X)_{\QQ}$ denote the Chow groups of $X$ (i.e. the groups of codimension $i$ algebraic cycles on $X$ with $\QQ$--coefficients, modulo rational equivalence). Let $A^i_{hom}(X)$ (and $A^i_{AJ}(X)$) denote the subgroup of homologically trivial (resp. Abel--Jacobi trivial) cycles.

The field of algebraic cycles is something of a mathematician's goldmine, with a wealth of open questions lying around for the picking.
 Many of these open questions are special cases or variants of Bloch's conjecture:
 
\begin{conjecture}[Bloch \cite{B}]\label{bloch1} Let $X$ be a smooth projective variety of dimension $n$. Let $\Gamma\in A^n(X\times X)$ be such that
  \[ \Gamma_\ast=0\colon\ \ \ H^i(X,\OO_X)\ \to\ H^i(X,\OO_X)\ \ \ \forall i>0\ .\]
  Then
  \[ \Gamma_\ast=0\colon\ \ \ A^n_{hom}(X)\ \to\ A^n(X)\ .\]
  \end{conjecture}
   
  \begin{conjecture}[Bloch \cite{B}]\label{bloch2} Let $X$ be a smooth projective variety of dimension $n$. Assume that
  \[  H^i(X,\OO_X)=0\ \ \ \forall i>0\ .\]
  Then
  \[  A^n_{}(X)\cong\QQ\ .\]
  \end{conjecture}
  
 The ``absolute version'' (conjecture \ref{bloch2}) is obtained from the ``relative version'' (conjecture \ref{bloch1}) by taking $\Gamma$ to be the diagonal. Conjecture \ref{bloch2} is famously open for surfaces of general type (cf. \cite{PW}, \cite{V8}, \cite{Gul} for some recent progress).

Let us now suppose that $X$ is a hyperk\"ahler variety (i.e., a projective irreducible holomorphic symplectic manifold \cite{Beau0}, \cite{Beau1}), say of dimension $2m$. Suppose there exists a non--symplectic automorphism $\sigma\in\aut(X)$ of order $k>m$. This implies that
  \[    \bigl( \sigma+\sigma^2 + \ldots +\sigma^k\bigr){}_\ast=0\colon\ \ \ H^i(X,\OO_X)\ \to\ H^{i}(X,\OO_X)\ \ \ \forall i>0\ .\]
  
 Conjecture \ref{bloch1} (applied to the correspondence $\Gamma=\sum_{j=1}^k \Gamma_{\sigma^j}\in A^{2m}(X\times X)$, where $\Gamma_f$ denotes the graph of an automorphism $f\in\aut(X)$) then predicts the following:

\begin{conjecture}\label{conjhk} Let $X$ be a hyperk\"ahler variety of dimension $2m$. Let $\sigma\in\aut(X)$ be an order $k$ non--symplectic automorphism, and assume $k>m$. Then
  \[  \begin{split}   \bigl( \sigma+\sigma^2 + \ldots +\sigma^k\bigr){}_\ast&=0\colon\ \ \ A^{2m}_{hom}(X)\ \to\ A^{2m}(X)\ ;\\
            \bigl( \sigma+\sigma^2 + \ldots +\sigma^k\bigr){}_\ast&=0\colon\ \ \ A^{2}_{AJ}(X)\ \to\ A^{2m}(X)\ .\\ 
            \end{split} \]
 \end{conjecture}
 
 (Here, the second statement follows from the first by applying the Bloch--Srinivas argument \cite{BS}.)
  
In \cite{nonsymp3}, this conjecture was verified for one family of hyperk\"ahler fourfolds, given by Fano varieties of lines on certain cubic fourfolds with an order $3$ non--symplectic automorphism. 
There exist three other families of cubic fourfolds with a polarized order $3$ automorphism inducing a non--symplectic automorphism on the Fano varieties (the four families are presented in \cite[Examples 6.4, 6.5, 6.6 and 6.7]{BCS}). The aim of the present note is to treat these $3$ remaining families.

 A first result is that conjecture \ref{conjhk} is true for one of the families:
 
  \begin{nonumbering}[=theorem \ref{main0}] Let $Y\subset\PP^5(\C)$ be a smooth cubic defined by an equation   
    \[  f(X_0,\ldots,X_4) + (X_5)^3=0\ ,\]
   where $f$ is homogeneous of degree $3$. Let $\sigma_Y\in\aut(Y)$ be the order $3$ automorphism induced by
    \[  \begin{split} \PP^5(\C)\ &\to\ \PP^5(\C)\ ,\\
                         [X_0:\ldots:X_5]\ &\mapsto\ [X_0:X_1:X_2:X_3: X_4:\nu X_5]\\
                         \end{split}\]
    (where $\nu$ is a $3$rd root of unity).

    Let $X=F(Y)$ be the Fano variety of lines in $Y$, and let $\sigma\in\aut(X)$ the non--symplectic automorphism induced by $\sigma_Y$.  
      Then
    \[ (\ide +\sigma+\sigma^2)_\ast \ A^4_{hom}(X)=0\ .\]
       \end{nonumbering}

Theorem \ref{main0} is proven using the theory of finite--dimensional motives \cite{Kim}. The argument is similar to that of \cite{nonsymp3}.

For the two other families, we prove a weak version of conjecture \ref{conjhk}:

  \begin{nonumbering}[=theorem \ref{main}] Let $(Y,\sigma_Y)$ be one of the following:

  \noindent
  (a)
  $Y\subset\PP^5(\C)$ is a smooth cubic fourfold defined by an equation
    \[ f(X_0,X_1,X_2)+ g(X_3,X_4)+ (X_5)^3 + X_5\bigl( X_3\ell_1(X_0,X_1,X_2)+X_4\ell_2(X_0,X_1,X_2)\bigr)=0\ ,\]
    where $f,g$ are homogeneous polynomials of degree $3$ and $\ell_1,\ell_2$ are linear forms. $\sigma_Y\in\aut(Y)$ is the order $3$ automorphism induced by
    \[  \begin{split} \PP^5(\C)\ &\to\ \PP^5(\C)\ ,\\
                         [X_0:\ldots:X_5]\ &\mapsto\ [X_0:X_1:X_2:\nu X_3: \nu X_4:\nu^2 X_5]\\
                         \end{split}\]
    (where $\nu$ is a $3$rd root of unity). 
    
  \noindent
  (b)
    $Y\subset\PP^5(\C)$ is a smooth cubic fourfold defined by an equation
    \[ \begin{split} X_2 f(X_0,X_1)+ X_3 g(X_0,X_1)+ (X_4)^2\ell_1(X_0,X_1) + X_4X_5\ell_2(X_0,X_1)+&\\ (X_5)^2\ell_3(X_0,X_1)+ X_4 h(X_2,X_3) +X_5 k(X_2,X_3)&=0\ ,\\
    \end{split}\]
    where $f,g,h,k$ are homogeneous polynomials of degree $2$ and $\ell_1,\ell_2$ are linear forms. $\sigma_Y\in\aut(Y)$ is the order $3$ automorphism induced by
    \[  \begin{split} \PP^5(\C)\ &\to\ \PP^5(\C)\ ,\\
                         [X_0:\ldots:X_5]\ &\mapsto\ [X_0:X_1:\nu X_2:\nu X_3: \nu^2 X_4:\nu^2 X_5]\\
                         \end{split}\]
    (where $\nu$ is a $3$rd root of unity). 
    
  Let $X=F(Y)$ be the Fano variety of lines in $Y$, and let $\sigma\in\aut(X)$ the non--symplectic automorphism induced by $\sigma_Y$.  
      Then
       \[ \begin{split}  (\ide +\sigma+\sigma^2)_\ast \ A^2_{(2)}(X)&=0\ ,\\
                               (\ide +\sigma+\sigma^2)_\ast \ A^4_{(2)}(X)&=0\ .\\                     
                               \end{split}        \]
       \end{nonumbering}

   Here, the notation $A^\ast_{(\ast)}(X)$ refers to the {\em Fourier decomposition\/} of the Chow ring of $X$ as constructed by Shen--Vial \cite{SV}, which plays an important role in this note. Theorem \ref{main} is proven by using Voisin's method of ``spread'', as developed in \cite{V0}, \cite{V1}. The argument is similar to that of \cite{LFu2} (which dealt with symplectic automorphisms on Fano varieties of cubic fourfolds), and that of \cite{inv}, \cite{inv2} (which dealt with anti--symplectic involutions on Fano varieties of cubic fourfolds). 
   
   I have not been able to prove the full conjecture \ref{conjhk} for the two families of theorem \ref{main}; the missing piece concerns the action on $A^4_{(4)}(X)$ (cf. remark \ref{problem} for discussion).

 As an immediate corollary, we find that Bloch's conjecture \ref{bloch2} is verified for the quotient of the Fano varieties in one family:
  
  \begin{nonumberingc}[=corollary \ref{triv}] Let $(X,\sigma)$ be as in theorem \ref{main0}, and let $Z:=X/<\sigma>$ be the quotient. Then
    \[  A^4_{hom}(Z)=0\ .\]
    \end{nonumberingc}
    
  Another corollary is the following property of the Chow ring of these Fano varieties (reminiscent of the Chow ring of $K3$ surfaces \cite{BV}):
  
  \begin{nonumberingc}[=corollary \ref{ring}] Let $X$ and $\sigma$ be as in theorem \ref{main0} or theorem \ref{main}. Let $a\in A^3(X)$ be a $1$--cycle of the form
  \[ a=\displaystyle\sum_{i=1}^r b_i\cdot D_i\ \ \ \in A^3(X)\ ,\]
  where $b_i\in A^2(X)^\sigma$ and $D_i\in A^1(X)_{}$. Then $a$ is rationally trivial if and only if $a$ is homologically trivial.
    \end{nonumberingc}
  

 \vskip0.6cm

\begin{convention} In this article, the word {\sl variety\/} will refer to a reduced irreducible scheme of finite type over $\C$. A {\sl subvariety\/} is a (possibly reducible) reduced subscheme which is equidimensional. 

{\bf All Chow groups will be with rational coefficients}: we will denote by $A_j(X)$ the Chow group of $j$--dimensional cycles on $X$ with $\QQ$--coefficients; for $X$ smooth of dimension $n$ the notations $A_j(X)$ and $A^{n-j}(X)$ are used interchangeably. 

The notations $A^j_{hom}(X)$, $A^j_{AJ}(X)$ will be used to indicate the subgroups of homologically trivial, resp. Abel--Jacobi trivial cycles.
For a morphism $f\colon X\to Y$, we will write $\Gamma_f\in A_\ast(X\times Y)$ for the graph of $f$.
The contravariant category of Chow motives (i.e., pure motives with respect to rational equivalence as in \cite{Sc}, \cite{MNP}) will be denoted $\MM_{\rm rat}$.



We will write $H^j(X)$ 
to indicate singular cohomology $H^j(X,\QQ)$.

Given an automorphism $\sigma\in\aut(X)$, we will write $A^j(X)^\sigma$ (and $H^j(X)^\sigma$) for the subgroup of $A^j(X)$ (resp. $H^j(X)$) invariant under 
$\sigma$.
\end{convention}

\section{Preliminaries}

\subsection{Quotient varieties}
\label{ssquot}

\begin{definition} A {\em projective quotient variety\/} is a variety
  \[ Z=X/G\ ,\]
  where $X$ is a smooth projective variety and $G\subset\hbox{Aut}(X)$ is a finite group.
  \end{definition}
  
 \begin{proposition}[Fulton \cite{F}]\label{quot} Let $Z$ be a projective quotient variety of dimension $n$. Let $A^\ast(Z)$ denote the operational Chow cohomology ring. The natural map
   \[ A^i(Z)\ \to\ A_{n-i}(Z) \]
   is an isomorphism for all $i$.
   \end{proposition}
   
   \begin{proof} This is \cite[Example 17.4.10]{F}.
      \end{proof}

\begin{remark} It follows from proposition \ref{quot} that the formalism of correspondences goes through unchanged for projective quotient varieties (this is also noted in \cite[Example 16.1.13]{F}). We may thus consider motives $(Z,p,0)\in\MM_{\rm rat}$, where $Z$ is a projective quotient variety and $p\in A^n(Z\times Z)$ is a projector. For a projective quotient variety $Z=X/G$, one readily proves (using Manin's identity principle) that there is an isomorphism
  \[  h(Z)\cong h(X)^G:=(X,\Delta_G,0)\ \ \ \hbox{in}\ \MM_{\rm rat}\ ,\]
  where $\Delta_G$ denotes the idempotent ${1\over \vert G\vert}{\sum_{g\in G}}\Gamma_g$.  
  \end{remark}

\subsection{Finite--dimensional motives}

We refer to \cite{Kim}, \cite{An}, \cite{J4}, \cite{MNP} for the definition of finite--dimensional motive. 
An essential property of varieties with finite--dimensional motive is embodied by the nilpotence theorem:

\begin{theorem}[Kimura \cite{Kim}]\label{nilp} Let $X$ be a smooth projective variety of dimension $n$ with finite--dimensional motive. Let $\Gamma\in A^n(X\times X)_{}$ be a correspondence which is numerically trivial. Then there is $N\in\NN$ such that
     \[ \Gamma^{\circ N}=0\ \ \ \ \in A^n(X\times X)_{}\ .\]
\end{theorem}

 Actually, the nilpotence property (for all powers of $X$) could serve as an alternative definition of finite--dimensional motive, as shown by Jannsen \cite[Corollary 3.9]{J4}.
Conjecturally, any variety has finite--dimensional motive \cite{Kim}. We are still far from knowing this, but at least there are quite a few non--trivial examples.

 \subsection{Spread}

  \begin{lemma}[Voisin \cite{V0}, \cite{V1}]\label{projbundle} Let $M$ be a smooth projective variety of dimension $n+1$, and $L$ a very ample line bundle on $M$. Let 
    \[ \pi\colon \YY\to B\]
    denote a family of hypersurfaces, where $B\subset\vert L\vert$ is a Zariski open.
      Let
   \[   p\colon \wt{\YY\times_B \YY}\ \to\ \YY\times_B \YY\]
   denote the blow--up of the relative diagonal. 
 Then $\wt{\YY\times_B \YY}$ is Zariski open in $V$, where $V$ is a projective bundle over $\wt{M\times M}$, the blow--up of $M\times M$ along the diagonal.
   \end{lemma} 
  
  \begin{proof} This is \cite[Proof of Proposition 3.13]{V0} or \cite[Lemma 1.3]{V1}. The idea is to define $V$ as
   \[  V:=\Bigl\{ \bigl((x,y,z),\sigma\bigr) \ \vert\ \sigma\vert_z=0\Bigr\}\ \ \subset\ \wt{M\times M}\times \vert L\vert\ .\]
   The very ampleness assumption ensures $V\to\wt{M\times M}$ is a projective bundle.
    \end{proof}

  This is used in the following key proposition: 
   
     \begin{proposition}[Voisin \cite{V1}]\label{voisin1} Assumptions as in lemma \ref{projbundle}. Assume moreover $M$ has trivial Chow groups. Let $R\in A^n(V)_{}$. Suppose that for all $b\in B$ one has
    \[ H^n(Y_b)_{prim}\not=0\ \ \ \ 
  \hbox{and}\ \ \ \ 
     R\vert_{\wt{Y_b\times Y_b}}=0\ \ \in H^{2n}(\wt{Y_b\times Y_b})\ .\]
   Then there exists $\delta\in A^n(M\times M)_{}$ such that
    \[     (p_b)_\ast \bigl(R\vert_{\wt{Y_b\times Y_b}}\bigr)= \delta\vert_{Y_b\times Y_b}  \ \ \in A^{n}({Y_b\times Y_b})_{}\]  
    for all $b\in B$. 
   (Here $p_b$ denotes the restriction of $p$ to $\wt{Y_b\times Y_b}$, which is the blow--up of $Y_b\times Y_b$ along the diagonal.)
    \end{proposition}

\begin{proof} This is \cite[Proposition 1.6]{V1}.
\end{proof}

 The following is an equivariant version of proposition \ref{voisin1}:
 
  \begin{proposition}[Voisin \cite{V1}]\label{voisin2} Let $M$ and $L$ be as in proposition \ref{voisin1}. Let $G\subset\aut(M)$ be a finite group. Assume the following:
  
  \noindent
  (\rom1) The linear system $\vert L\vert^G:=\PP\bigl( H^0(M,L)^G\bigr)$ has no base--points, and the locus of points in $\wt{M\times M}$ parametrizing triples $(x,y,z)$ such that the length $2$ subscheme $z$ imposes only one condition on $\vert L\vert^G$ is contained in the union of (proper transforms of) graphs of non--trivial elements of $G$, plus some loci of codimension $>n+1$.
  
  \noindent
  (\rom2) Let $B\subset\vert L\vert^G$ be the open parametrizing smooth hypersurfaces, and let $Y_b\subset M$ be a hypersurface for $b\in B$ general. There is no non--trivial relation
   \[ {\displaystyle\sum_{g\in G}} c_g \Gamma_g +\gamma=0\ \ \ \hbox{in}\ H^{2n}(Y_b\times Y_b)\ ,\]
   where $c_g\in\QQ$ and $\gamma$ is a cycle in $\ima\bigl( A^n(M\times M)\to A^n(Y_b\times Y_b)\bigr)$.
   
   Let $R\in A^n(\YY\times_B \YY)$ be such that
     \[  R\vert_{{Y_b\times Y_b}}=0\ \ \in H^{2n}({Y_b\times Y_b})\ \ \ \forall b\in B\ .\]
    Then there exists $\delta\in A^n(M\times M)_{}$ such that
    \[     R\vert_{{Y_b\times Y_b}}= \delta\vert_{Y_b\times Y_b}  \ \ \in A^{n}({Y_b\times Y_b})\ \ \ \forall b\in B\ .\]  
  \end{proposition} 
  
 \begin{proof} This is not stated verbatim in \cite{V1}, but it is contained in the proof of \cite[Proposition 3.1 and Theorem 3.3]{V1}. We briefly review the argument.
 One considers
   \[  V:=\Bigl\{ \bigl((x,y,z),\sigma\bigr) \ \vert\ \sigma\vert_z=0\Bigr\}\ \ \subset\ \wt{M\times M}\times \vert L\vert^G\ .\]   
   The problem is that this is no longer a projective bundle over $\wt{M\times M}$. However, as explained in the proof of \cite[Theorem 3.3]{V1}, hypothesis (\rom1) ensures that one 
   can obtain a projective bundle after blowing up the graphs $\Gamma_g, g\in G$ plus some loci of codimension $>n+1$. Let $M^\prime\to\wt{M\times M}$ denote the result of these blow--ups, and let $V^\prime\to M^\prime$ denote the projective bundle obtained by base--changing. 

Analyzing the situation as in \cite[Proof of Theorem 3.3]{V1}, one obtains
   \[ R\vert_{Y_b\times Y_b} =R_0\vert_{Y_b\times Y_b}+ {\displaystyle\sum_{g\in G}} \lambda_g \Gamma_g\ \ \ \hbox{in}\ A^n(Y_b\times Y_b) \ ,\]
   where $R_0\in A^n(M\times M)$ and $\lambda_g\in\QQ$ (this is \cite[Equation (15)]{V1}).
   By assumption, $R\vert_{Y_b\times Y_b}$ is homologically trivial. Using hypothesis (\rom2), this implies that all $\lambda_g$ have to be $0$.   
     \end{proof}

\subsection{MCK decomposition}

\begin{definition}[Murre \cite{Mur}] Let $X$ be a smooth projective variety of dimension $n$. We say that $X$ has a {\em CK decomposition\/} if there exists a decomposition of the diagonal
   \[ \Delta_X= \pi_0+ \pi_1+\cdots +\pi_{2n}\ \ \ \hbox{in}\ A^n(X\times X)\ ,\]
  such that the $\pi_i$ are mutually orthogonal idempotents and $(\pi_i)_\ast H^\ast(X)= H^i(X)$.
  
  (NB: ``CK decomposition'' is shorthand for ``Chow--K\"unneth decomposition''.)
\end{definition}

\begin{remark} The existence of a CK decomposition for any smooth projective variety is part of Murre's conjectures \cite{Mur}, \cite{J2}. 
\end{remark}

\begin{definition}[Shen--Vial \cite{SV}] Let $X$ be a smooth projective variety of dimension $n$. Let $\Delta_X^{sm}\in A^{2n}(X\times X\times X)$ be the class of the small diagonal
  \[ \Delta_X^{sm}:=\bigl\{ (x,x,x)\ \vert\ x\in X\bigr\}\ \subset\ X\times X\times X\ .\]
  An {\em MCK decomposition\/} is a CK decomposition $\{\pi^X_i\}$ of $X$ that is {\em multiplicative\/}, i.e. it satisfies
  \[ \pi^X_k\circ \Delta_X^{sm}\circ (\pi^X_i\times \pi^X_j)=0\ \ \ \hbox{in}\ A^{2n}(X\times X\times X)\ \ \ \hbox{for\ all\ }i+j\not=k\ .\]
  
 (NB: ``MCK decomposition'' is shorthand for ``multiplicative Chow--K\"unneth decomposition''.) 
  
 A {\em weak MCK decomposition\/} is a CK decomposition $\{\pi^X_i\}$ of $X$ that satisfies
    \[ \Bigl(\pi^X_k\circ \Delta_X^{sm}\circ (\pi^X_i\times \pi^X_j)\Bigr){}_\ast (a\times b)=0 \ \ \ \hbox{for\ all\ } a,b\in\ A^\ast(X)\ .\]
  \end{definition}
  
  \begin{remark} The small diagonal (seen as a correspondence from $X\times X$ to $X$) induces the {\em multiplication morphism\/}
    \[ \Delta_X^{sm}\colon\ \  h(X)\otimes h(X)\ \to\ h(X)\ \ \ \hbox{in}\ \MM_{\rm rat}\ .\]
 Suppose $X$ has a CK decomposition
  \[ h(X)=\bigoplus_{i=0}^{2n} h^i(X)\ \ \ \hbox{in}\ \MM_{\rm rat}\ .\]
  By definition, this decomposition is multiplicative if for any $i,j$ the composition
  \[ h^i(X)\otimes h^j(X)\ \to\ h(X)\otimes h(X)\ \xrightarrow{\Delta_X^{sm}}\ h(X)\ \ \ \hbox{in}\ \MM_{\rm rat}\]
  factors through $h^{i+j}(X)$.
  
  If $X$ has a weak MCK decomposition, then setting
    \[ A^i_{(j)}(X):= (\pi^X_{2i-j})_\ast A^i(X) \ ,\]
    one obtains a bigraded ring structure on the Chow ring: that is, the intersection product sends $A^i_{(j)}(X)\otimes A^{i^\prime}_{(j^\prime)}(X) $ to  $A^{i+i^\prime}_{(j+j^\prime)}(X)$.
    
      It is expected (but not proven !) that for any $X$ with a weak MCK decomposition, one has
    \[ A^i_{(j)}(X)\stackrel{??}{=}0\ \ \ \hbox{for}\ j<0\ ,\ \ \ A^i_{(0)}(X)\cap A^i_{hom}(X)\stackrel{??}{=}0\ ;\]
    this is related to Murre's conjectures B and D, that have been formulated for any CK decomposition \cite{Mur}.

  The property of having an MCK decomposition is severely restrictive, and is closely related to Beauville's ``(weak) splitting property'' \cite{Beau3}. For more ample discussion, and examples of varieties with an MCK decomposition, we refer to \cite[Section 8]{SV}, as well as \cite{V6}, \cite{SV2}, \cite{FTV}.
    \end{remark}

In what follows, we will make use of the following: 

\begin{theorem}[Shen--Vial \cite{SV}]\label{sv} Let $Y\subset\PP^5(\C)$ be a smooth cubic fourfold, and let $X:=F(Y)$ be the Fano variety of lines in $Y$. There exists a self--dual CK decomposition $\{\Pi^X_i\}$ for $X$, and 
  \[ (\Pi^X_{2i-j})_\ast A^i(X) = A^i_{(j)}(X)\ ,\]
  where the right--hand side denotes the splitting of the Chow groups defined in terms of the Fourier transform as in \cite[Theorem 2]{SV}. Moreover, we have
  \[ A^i_{(j)}(X)=0\ \ \ \hbox{if\ }j<0\ \hbox{or\ }j>i\ \hbox{or\ } j\ \hbox{is\ odd}\ .\]
  
  In case $Y$ is very general, the Fourier decomposition $A^\ast_{(\ast)}(X)$ forms a bigraded ring, and hence
  $\{\Pi^X_i\}$ is a weak MCK decomposition.
    \end{theorem}

\begin{proof} This is a summary of results in \cite{SV} (cf. also \cite[Theorem 2.11]{nonsymp3} for precise attributions).


    \end{proof}

\begin{remark}\label{pity} Unfortunately, it is not yet known that the Fourier decomposition of \cite{SV} induces a bigraded ring structure on the Chow ring for {\em all\/} Fano varieties of smooth cubic fourfolds. For one thing, it has not yet been proven that $A^2_{(0)}(X)\cdot A^2_{(0)}(X)\subset A^4_{(0)}(X)$ (cf. \cite[Section 22.3]{SV} for discussion).
\end{remark}

 \subsection{Relative CK decomposition}

\begin{notation}\label{not}  Let
   \[ \YY\ \to\ B \]
   denote the universal family of smooth cubic fourfolds $Y_b$. Here $B$ is a Zariski open in the parameter space 
   $\PP H^0(\PP^5,\OO_{\PP^5}(3))$. Let
   \[ \XX := \{  (\ell, b)\ \vert\ \ell \subset Z_b \}\ \ \ \subset\ G(1,5)\times B \]
   be the corresponding family of Fano varieties of lines. (Here $G(1,5)$ denotes the Grassmannian of lines in $\PP^5$.) A fibre $X_b$ of $\XX\to B$ is a Fano variety of lines on a cubic $Y_b$.
\end{notation}

 \begin{proposition}\label{relck} Let $\XX\to B$ be as above. There exist relative correspondences
   \[ \Pi_i^\XX\ \ \ \in A^4(\XX\times_B \XX)\ \ \ (i=2,6)\ ,\]
   with the property that for each $b\in B$, one has
   \[ \begin{split}  \bigl((\Pi_2^\XX)\vert_{X_b\times X_b}\bigr){}_\ast &= (\Pi_2^{X_b})_\ast\colon\ \ \  A^2(X_b)\ \to\ A^2(X_b) \ ,\\
                            \bigl((\Pi_6^\XX)\vert_{X_b\times X_b}\bigr){}_\ast &= (\Pi_6^{X_b})_\ast\colon\ \ \  A^4(X_b)\ \to\ A^4(X_b) \ .\\
                            \end{split} \]
Here $\Pi_i^{X_b}$  is the Chow--K\"unneth decomposition of theorem \ref{sv}(\rom1).
 \end{proposition}
 
 \begin{proof} The main point is that the Shen--Vial cycle $L\in A^2(X_b\times X_b)$ furnishing the Fourier decomposition \cite{SV} exists relatively: this is because by definition
   \[ L:= {1\over 3}(g_1^2+{3\over 2}g_1 g_2+g_2^2-c_1-c_2)-I\ \ \ \in A^2(X_b\times X_b)\ \]
   \cite[Equation (107)]{SV}. Here, $g:=-c_1(\EE_2)\in A^1(X_b)$ and $c:=c_2(\EE_2)\in A^2(X_b)$ (and $\EE_2$ is the restriction of the rank $2$ tautological bundle on the Grassmannian), and $g_i:=(p_i)^\ast(g)$, $c_i:=(p_i)^\ast(c)$ (where $p_i\colon X_b\times X_b\to X_b$ is projection on the $i$th factor), and $I\subset X_b\times X_b$ is the incidence correspondence.
   Since $g_i, c_i$ and $I$ obviously exist relatively, the same goes for $L$, i.e. there exists a relative correspondence
   \[ \LL\ \ \ \in A^2(\XX\times_B \XX) \]
   with the property that for any $b\in B$ the restriction
   \[ \LL\vert_{X_b\times X_b}\ \ \ \in A^2(X_b\times X_b) \]
   is the Shen--vial class $L$ of \cite{SV}. This implies that the class $\ell\in A^2(X_b)$ of \cite{SV} (mentioned in theorem \ref{mult}(\rom2) below) also exists relatively: it is defined as 
    \[  \ell:=  (i_\Delta)^\ast(\LL)  \ \ \ \in A^2(\XX)\ ,\]
    where $i_\Delta\colon \XX\to \XX\times_B \XX$ denotes the embedding along the relative diagonal. (NB: this makes sense because $i_\Delta$ is a regular embedding.) Next, the classes $\ell_i:=(p_i)^\ast(\ell)\in A^2(X_b\times X_b)$ of \cite{SV} also exist relatively.
    
  Armed with these facts, let us inspect the construction of the $\{\Pi_i^{X_b}\}$ in \cite[Theorem 3.3]{SV}. 
  
    
    As a first approach towards the construction of $\Pi^{X_b}_2$ and $\Pi^{X_b}_6$, Shen--Vial define
    \[ p_b:={1\over 25} L\cdot \ell_2\ \ \ \in A^4(X_b\times X_b)\ .\]
    Again, by the above remarks the cycle $p_b$ exists relatively (i.e. there is $\pp\in A^4(\XX\times_B \XX)$ which restricts to $p_b$ on each fibre). 
   We define $\Pi_6^\XX:=\pp\in A^4(\XX\times_B \XX)$. This does the job, for it is shown in \cite[Proof of Theorem 3.3]{SV} that $(p_b)_\ast$ acts as the identity on $A^4_{(2)}(X_b)$ and acts as $0$ on $A^4_{(0)}(X_b)\oplus A^4_{(4)}(X_b)$. 
   
   We define $\Pi_2^\XX$ as the transpose $\Pi_2^\XX:={}^t \Pi_6^\XX$. This does the job, for it is shown in loc. cit. that $(p_b)^\ast$ acts as the identity on $A^2_{(2)}(X_b)$ and acts as $0$ on $A^2_{(0)}(X_b)$.
  \end{proof}

\subsection{Refined CK decomposition}

\begin{theorem}[]\label{pi20} Let $Y\subset\PP^5(\C)$ be a smooth cubic fourfold, and let $X:=F(Y)$ be the Fano variety of lines in $Y$. 

 Then $X$ has a CK decomposition $\{\Pi^X_i\}$. Moreover, there exists a further splitting
  \[ \Pi^X_2 = \pi^X_{2,0} + \pi^X_{2,1}\ \ \ \hbox{in}\ A^4(X\times X)\ ,\]
  where $\pi^X_{2,1}$ is supported on $C\times D\subset X\times X$, where $C$ and $D$ are a curve, resp. a divisor on $X$.
  
  The action on cohomology verifies
  \[ (\pi^X_{2,0})_\ast H^\ast(X) = H^2_{tr}(X)\ ,\]
  where $H^2_{tr}(X)\subset H^2(X)$ is defined as the orthogonal complement of $NS(X)$ with respect to the Beauville--Bogomolov form. 
  
  The action on Chow groups verifies
  \[  (\pi^X_{2,0})_\ast A^2(X) = A^2_{(2)}(X)\ .\]
  \end{theorem}

  \begin{proof} The existence of a CK decomposition is theorem \ref{sv}. 
  
  To define the refined decomposition, let $\pi_4^Y\in A^4(Y\times Y)$ be a CK projector of $Y$. Pedrini \cite[Section 4]{Ped} has constructed a decomposition
   \[  \pi_4^Y = \pi^Y_{4,tr} + \pi^Y_{4,alg}\ \ \ \hbox{in}\ A^4(Y\times Y)\ ,\]
   where $\pi^Y_{4,tr}, \pi^Y_{4,alg}$ are mutually orthogonal idempotents, and
   \[ (\pi^Y_{4,tr})_\ast H^\ast(Y) = H^4_{tr}(Y)\ ,\]
   where $H^4_{tr}(Y)$ is by definition the orthogonal complement (under the cup product) of $N^2 H^4(Y)$.
  
  To obtain a similar splitting for $X$, one uses the following: let $P\in A^3(X\times Y)$ be the universal family of lines. Then
   \[ ({}^t P)_\ast\colon\ \ \ H^4(Y)\ \to\ H^2(X) \]
   is an isomorphism, the {\em Abel--Jacobi isomorphism\/} \cite{BD}. Moreover, there is a correspondence $Q\in A^5(X\times Y)$ inducing the inverse isomorphism, i.e.
    \[  \Pi^X_2 =     {}^t P \circ \pi_4^Y \circ Q \ \ \ \hbox{in}\ H^8(X\times X)\ .\]
    (An explicit formula for $Q$ is $Q=-{1\over 6} P\circ \Gamma_{g^2}$, where $g\in A^1(X)$ is the Pl\"ucker polarization, and $\Gamma_{g^2}\in A^6(X\times X)$ is the correspondence acting as multiplication with $g^2$. This follows from \cite[Proof of Proposition 6]{BD}.)
    
    We now define
     \[  \begin{split}  \pi^X_{2,1} &:=   {}^t P \circ \pi_{4,alg}^Y \circ Q  \ ,\\
                               \pi^X_{2,0}&:=  \Pi^X_2-\pi^X_{2,1}   \ \ \ \ \ \ \in A^4(X\times X)\ .\\
                               \end{split}\]
                               
            Lieberman's lemma \cite[Lemma 3.3]{V3} gives an equality
        \[ \pi^X_{2,1}= ({}^t Q\times {}^t P)_\ast \pi_{4,alg}^Y\ \ \ \hbox{in}\ A^4(X\times X)  \ .\]  
        This implies that $\pi^X_{2,1}$ is supported on $C\times D\subset X\times X$, where $C\subset X$ is the curve
        \[ p_X \bigl( \hbox{Supp} \bigl( (p_Y)^\ast(V)\cdot ({}^t Q)\bigr)\bigr)\ \ \ \subset X\ ,\]
        and $D\subset X$ is the divisor
        \[ p_X \bigl( \hbox{Supp} \bigl( (p_Y)^\ast(V)\cdot ({}^t P)\bigr)\bigr)\ \ \ \subset X\ .\]
        
        For reasons of dimension, $\pi^X_{2,1}$ acts trivially on $A^2(X)$, and so
        \[ (\pi^X_{2,0})_\ast A^2(X) =  (\Pi^X_2)_\ast A^2(X)=   A^2_{(2)}(X)\ .\]

            Since (by construction) the projector $\pi^Y_{4,alg}$ is supported on $V\times V$, for some closed codimension $2$ subvariety $V\subset Y$,  
            
           Finally, it is known that the Abel--Jacobi isomorphism is an isometry \cite{BD}, and so 
           \[ ({}^t P)_\ast H^4_{tr}(Y)= H^2_{tr}(X)\ . \]
                           
                                \end{proof}

         \begin{remark} We do {\em not\/} claim that the correspondences $\pi^X_{2,0}, \pi^X_{2,1}$ are mutually orthogonal and idempotent ! This can presumably be achieved, but we do not need this.
           \end{remark}

\subsection{A multiplicative result}

Let $X$ be the Fano variety of lines on a smooth cubic fourfold. As we have seen (theorem \ref{sv}), the Chow ring of $X$ splits into pieces $A^i_{(j)}(X)$.
The magnum opus \cite{SV} contains a detailed analysis of the multiplicative behaviour of these pieces. Here are the relevant results we will be needing:

\begin{theorem}[Shen--Vial \cite{SV}]\label{mult} Let $Y\subset\PP^5(\C)$ be a smooth cubic fourfold, and let $X:=F(Y)$ be the Fano variety of lines in $Y$. 

\noindent
(\rom1) There exists $\ell\in A^2_{(0)}(X)$ such that intersecting with $\ell$ induces an isomorphism
  \[ \cdot\ell\colon\ \ \ A^2_{(2)}(X)\ \xrightarrow{\cong}\ A^4_{(2)}(X)\ .\]
  The inverse isomorphism is given by 
   \[ {1\over 25} L_\ast\colon\ \ \ A^4_{(2)}(X)\ \xrightarrow{\cong}\ A^2_{(2)}(X)\ ,\]
  where $L\in A^2(X\times X)$ is the class defined in \cite[Equation (107)]{SV}.
  
\noindent
(\rom2) Intersection product induces a surjection
  \[ A^2_{(2)}(X)\otimes A^2_{(2)}(X)\ \twoheadrightarrow\ A^4_{(4)}(X)\ .\]
\end{theorem} 
     
 \begin{proof} Statement (\rom1) is \cite[Theorem 4]{SV}. Statement (\rom2) is \cite[Proposition 20.3]{SV}.
  \end{proof}

    The following is a reformulation of theorem \ref{mult}(\rom1). 
    
   \begin{proposition}\label{withI} Let $Y\subset\PP^5(\C)$ be a smooth cubic fourfold, and let $X$ be the Fano variety of lines in $Y$.
   Let $I\in A^2(X\times X)$ be the incidence correspondence, and let $g=-c_1(\EE_2)\in A^1(X)$ be the Pl\"ucker polarization. Then
   \[ \cdot g^2\colon\ \ \ A^2_{(2)}(X)\ \xrightarrow{}\ A^4_{(2)}(X)\ \]
   is an isomorphism. The inverse isomorphism is given by
    \[ -{1\over 6} I_\ast\colon\ \ \ A^4_{(2)}(X)\ \to\ A^2_{(2)}(X)\ .\]
        \end{proposition} 
        
     \begin{proof} This is implicit in the arguments of \cite{SV}. For any $a\in A^4_{hom}(X)$, there is equality
      \[ \ell\cdot L_\ast(a)=-{25\over 6} g^2\cdot I_\ast(a)\ \ \ \hbox{in}\ A^4(X)\ .\]
      (This follows from \cite[Equations (107) and (108)]{SV}, cf. the proof of \cite[Proposition 19.4]{SV}.)
      But for $a\in A^4_{(2)}(X)$, we know (theorem \ref{mult}(\rom2)) that $\ell\cdot L_\ast(a)=25a$, and so for any $a\in A^4_{(2)}(X)$ we get an equality
      \[ a= -{1\over 6} g^2\cdot I_\ast(a)\ \ \ \hbox{in}\ A^4(X)\ .\]
      Applying $I_\ast$ to this equality, we obtain an equality (for any $a\in A^4_{(2)}(X)$) 
      \begin{equation}\label{Iast} I_\ast(a) =-{1\over 6} I_\ast(g^2\cdot I_\ast(a))\ \ \ \hbox{in}\ A^2(X)\ .\end{equation}
      But we know that
       \[   A^2_{(2)}(X) = I_\ast A^4_{hom}(X)= I_\ast A^4_{(2)}(X)\ .\]
       (Here, the first equality is \cite[Proof of Proposition 21.10]{SV}, and the second equality follows from the fact that $I_\ast A^4_{(4)}(X)=0$ 
       \cite[Theorem 20.5]{SV}.)
     Equation (\ref{Iast}) thus becomes the statement that for any $b\in A^2_{(2)}(X)$ there is equality
      \[ b= -{1\over 6} I_\ast(g^2\cdot b)\ \ \ \hbox{in}\ A^2(X) \ .  \]
      This proves the proposition.
            \end{proof}

  \begin{corollary}\label{decomp} Let $Y\subset\PP^5(\C)$ be a smooth cubic fourfold, and let $X$ be the Fano variety of lines in $Y$. There exist correspondences
  $P\in A^3(X\times Y)$, $\Psi\in A^5(Y\times X)$ such that
    \[ \begin{split} ({}^t P\circ {}^t \Psi)_\ast&=\ide\colon\ \ \ A^2_{(2)}(X)\ \to\ A^2_{(2)}(X)\ ,\\
                           ( {} \Psi\circ P)_\ast&=\ide\colon\ \ \ A^4_{(2)}(X)\ \to\ A^4_{(2)}(X)\ .\\ 
                      \end{split}\]   
         
         Moreover, if $\YY\to B$ and $\XX\to B$ denote the universal families as in notation \ref{not}, there exist relative correspondences $\PPP\in A^3(\XX\times_B \YY)$ and $\Psi\in A^5(\YY\times_B \XX)$ with the above property on each fibre: i.e., for each $b\in B$ we have
         \[ \begin{split} \bigl( ({}^t \PPP\circ {}^t \Psi)\vert_{X_b\times X_b}\bigr){}_\ast   &=(\Pi_2^{X_b})_\ast\colon\ \ \ A^2(X_b)\ \to\ A^2(X_b)\ ,\\   
\bigl( (\Psi\circ \PPP)\vert_{X_b\times X_b}\bigr){}_\ast   &=(\Pi_6^{X_b})_\ast\colon\ \ \ A^4(X_b)\ \to\ A^4(X_b)\ .\\
    \end{split}\]  
                           
  \end{corollary}

  \begin{proof} The correspondence $P\in A^3(X\times Y)$ is defined as the universal family of lines on $Y$. Letting $I\subset X\times X$ denote the incidence correspondence, we have
  \[ I={}^t P\circ P\ \ \ \hbox{in}\ A^2(X\times X) \]
  \cite[Lemma 17.2]{SV}. 
  
  Proposition \ref{withI} states that the composition
    \[   A^2_{(2)}(X)\ \xrightarrow{\cdot g^2}\ A^4_{(2)}(X)   \ \xrightarrow{-{1\over 6}({}^t P\circ P)_\ast}\ A^2_{(2)}(X) \]
    is the identity, in other words
    \[       \bigl( -{1\over 6} {}^t P\circ P\circ \Gamma_{g^2}\bigr){}_\ast=\ide\colon\ \ \  A^2_{(2)}(X)\ \to\ A^2_{(2)}(X)  \ .\]
    (Here $\Gamma_{g^2}\in A^6(X\times X)$ can be defined as ${1\over d}\, {}^t \Gamma_\tau\circ\Gamma_\tau$, where $\tau\colon R\to X$ denotes the inclusion of a smooth complete intersection of class $d g^2$).
    Defining 
     \[ {}^t \Psi:= -{1\over 6} P\circ \Gamma_{g^2}\ \ \ \in\ A^5(X\times Y)\ ,\]
     we obtain the first equality of corollary \ref{decomp}. 
     
     For the second equality, it suffices to take the transpose of the first equality, since we know that $\Pi^X_6={}^t \Pi^X_2$ (theorem \ref{sv}).
   
   As to the ``moreover'' part of the corollary: obviously both $P$ and $\Gamma_{g^2}$ exist relatively, and so the same goes for $\Psi$.

    \end{proof}

 \section{A family of triple covers}
 
 In this section, we consider cubic fourfolds that are triple covers over cubic threefolds. The main result here is as follows:
 
 \begin{theorem}\label{main0} Let $Y\subset\PP^5(\C)$ be a smooth cubic fourfold defined by an equation
    \[ f(X_0,\ldots,X_4)+ (X_5)^3=0\ ,\]
    where $f$ is a homogeneous polynomial of degree $3$. Let $X=F(Y)$ be the Fano variety of lines in $Y$. 
    Let $\sigma\in\aut(X)$ be the (non--symplectic) order $3$ automorphism induced by
    \[  \begin{split} \PP^5(\C)\ &\to\ \PP^5(\C)\ ,\\
                         [X_0:\ldots:X_5]\ &\mapsto\ [X_0:X_1:X_2:X_3: X_4:\nu X_5]\\
                         \end{split}\]
    (where $\nu$ is a $3$rd root of unity). 
    
    Then
    \[ (\ide +\sigma+\sigma^2)_\ast \ A^4_{hom}(X)=0\ .\]
 
 \end{theorem}

  \begin{proof}  
 (The family of Fano varieties of theorem \ref{main0} is described in \cite[Example 6.4]{BCS}, from which I learned that the automorphism $\sigma$ is 
 non--symplectic.)
 
 We will use the Fourier decomposition $A^\ast_{(\ast)}(X)$ of the Chow ring of $X$ (theorem \ref{sv}). Indeed, to prove theorem \ref{main}, it suffices to prove
 the following statement:
 
 \begin{equation}\label{state}   (\ide +\sigma+\sigma^2)_\ast \ A^i_{(j)}(X)=0\ \ \ \hbox{for}\ (i,j)\in\{(2,2),(4,2),(4,4)\}\ .   \end{equation}
  
  Indeed, we observe that
    \[ \begin{split}  A^4_{hom}(X)&= A^4_{(2)}(X)\oplus A^4_{(4)}(X)\ ,\\
                       \end{split}\]
            \cite{SV}, and so (\ref{state}) implies theorem \ref{main}.
             
 Let us now prove (\ref{state}).
                            
   In a first step of the proof, we show that the automorphism $\sigma$ respects the Fourier decomposition of the Chow ring:
 
 \begin{proposition}\label{compat} Let $X$ and $\sigma$ be as in theorem \ref{main0}. Let $A^\ast_{(\ast)}(X)$ be the Fourier decomposition (theorem \ref{sv}).
 Then
   \[ \sigma_\ast \, A^i_{(j)}(X)\ \subset\ A^i_{(j)}(X)\ \ \ \forall (i,j)\ .\]
   Equivalently, if $\{\Pi_j^X\}$ is a CK decomposition as in theorem \ref{sv}, then we have
    \[ \sigma_\ast (\Pi_j^X)_\ast   = (\Pi_j^X)_\ast \sigma_\ast  (\Pi_j^X)_\ast\colon\ \ \ A^i(X)\ \to\ A^i(X)\ \ \ \forall (i,j)\ .\]  
     \end{proposition}
  
  \begin{proof} (NB: The proof actually works for any $(X,\sigma)$, where $X=F(Y)$ is the Fano variety is of lines on a smooth cubic fourfold $Y$, and $\sigma\in\aut(X)$ is a polarized automorphism, i.e. induced by an automorphism of $Y$.)
  
  We start by recalling that the Shen--Vial cycle $L\in A^2(X\times X)$ (furnishing the Fourier decomposition of \cite{SV}) is defined as
   \begin{equation}\label{def} L:= {1\over 3}(g_1^2+{3\over 2}g_1 g_2+g_2^2-c_1-c_2)-I\ \ \ \in A^2(X\times X)\ \end{equation}
   \cite[Equation (107)]{SV}. Here $I\subset X\times X$ is the incidence correspondence, and $g_i, c_i$ are defined as follows:
     \[ \begin{split} g&:=-c_1(\EE_2)\ \ \ \in\ A^1(X)\ ,\\
                        c&:=c_2(\EE_2)\ \ \ \in\ A^2(X)\ ,\\
                        g_i&:= (p_i)^\ast(g)  \ \ \ \in\ A^1(X\times X)\ \ \ (i=1,2)\ ,\\
                        c_i&:= (p_i)^\ast(c)  \ \ \ \in\ A^2(X\times X)\ \ \ (i=1,2)\ ,\\ 
                       \end{split}\] 
    where $\EE_2$ is the rank $2$ vector bundle coming from the tautological bundle on the Grassmannian, and $p_i\colon X\times X\to X$ denote the two projections.

  We now observe that
   \[ (\sigma\times\sigma)^\ast(I) = I\ \ \ \hbox{in}\ A^2(X\times X)\ ,\]
   because a point $y\in Y$ is contained in a line $\ell\in X$ if and only if $\sigma^\ast(y)$ is contained in $\sigma^\ast(\ell)$. 
       
    Next, we note that the automorphism $\sigma\in\aut(X)$ comes from an automorphism $\sigma_\PP$ of $\PP^5(\C)$, hence in particular $\sigma$ is the restriction of an automorphism $\sigma_{Gr}$ of the Grassmannian $Gr=G(1,5)$ of lines in $\PP^5$. As $\sigma_\PP$ is linear, $(\sigma_{Gr})$ fixes the tautological vector bundle $\EE_2$ on $Gr$, and so 
    \[ (\sigma\times\sigma)^\ast(c_i)=c_i\ ,\ \ \ (\sigma\times\sigma)^\ast(g_i)=g_i\ \ \ \hbox{in}\ A^\ast(X\times X)\ .\]
    
  Applying $(\sigma\times\sigma)^\ast$ to equation (\ref{def}), it follows that
   \[ (\sigma\times\sigma)^\ast (L) = L\ \ \ \hbox{in}\ A^2(X\times X)\ .\]
     Using Lieberman's lemma 
   \cite[Lemma 3.3]{V3}, this translates into the equality
   \[ \Gamma_\sigma\circ 
   L \circ \Gamma_{\sigma^{-1}} = L\ \ \ \hbox{in}\ A^2(X\times X)\ ,\]
  or equivalently
   \[ \Gamma_\sigma\circ L = L\circ \Gamma_\sigma\ \ \ \hbox{in}\ A^2(X\times X)\ . \]
   This implies that also
   \[ \Gamma_\sigma\circ L^r = L^r\circ \Gamma_\sigma\ \ \ \hbox{in}\ A^{2r}(X\times X)\ \ \ \forall r\in\NN\ . \]
   
    As the Fourier transform $\FF\colon A^\ast(X)\to A^\ast(X)$ of \cite{SV} is defined by means of a polynomial in $L$, it follows that
    \[ \FF(\sigma^\ast(a))=\sigma^\ast \FF(a)\ \ \ \forall a\in A^\ast(X)\ ,\]
    proving the proposition.
          \end{proof}

  The second step of the proof is to ascertain that $X$ has finite--dimensional motive:
  
  \begin{proposition}\label{findim0} Let $Y\subset\PP^5(\C)$ and $X=F(Y)$ be as in theorem \ref{main0}. Then $Y$ and $X$ have finite--dimensional motive.
  \end{proposition}
  
  \begin{proof} 
  The finite--dimensionality of $Y$ is proven in \cite{fam}.
    The main result of \cite{fano} now implies that the Fano variety $X=F(Y)$ also has finite--dimensional motive.            
  \end{proof} 
 
 The third step of the proof is to show the desired statement for $A^2_{(2)}(X)$, i.e. we now prove that
   \begin{equation}\label{22}    (\ide +\sigma+\sigma^2)_\ast \ A^2_{(2)}(X)=0\ .  \end{equation}
   
   In order to do so, let us abbreviate
   \[ \Delta_G:= \Delta_X +\Gamma_\sigma +\Gamma_{\sigma\circ\sigma}\ \ \ \in\ A^4(X\times X)\ .\]
   Since the action of $\sigma$ is non--symplectic \cite[Example 6.5 and Lemma 6.2]{BCS}, we have that
   \[ (\Delta_G)_\ast=0\colon\ \ \ H^2(X,\OO_X)\ \to\ H^2(X,\OO_X)\ .\] 
   Using the Lefschetz $(1,1)$--theorem, we see that
    \[  \Delta_G\circ \Pi_2^X =\gamma \ \ \ \hbox{in}\ H^8(X\times X)\ ,\]
    where $\gamma$ is some cycle supported on $D\times D\subset X\times X$, for some divisor $D\subset X$. In other words, the correspondence
     \[ \Gamma:=  \Delta_G\circ \Pi_2^X - \gamma  \ \ \ \in A^4(X\times X) \]
     is homologically trivial. But then (since $X$ has finite-dimensional motive) there exists $N\in\NN$ such that
     \[  \Gamma^{\circ N}=0\ \ \ \hbox{in}\ A^4(X\times X)\ .\]
     Upon developing this expression, one finds an equality    
     \[ \Gamma^{\circ N}=  (\Delta_G\circ \Pi_2^X)^{\circ N} +  \gamma^\prime=0\ \ \ \hbox{in}\ A^4(X\times X)\ ,\]
    where $\gamma^\prime$ is supported on $D\times D\subset X\times X$. In particular, $\gamma^\prime$ acts trivially on $A^2_{(2)}(X)\subset A^2_{AJ}(X)$, and so
      \[   \bigl((\Delta_G\circ \Pi_2^X)^{\circ N}\bigr){}_\ast=0\colon\ \ \ A^2_{(2)}(X)\ \to\ A^2(X)\ .\]
      Proposition \ref{compat} (combined with the fact that $\Delta_G$ and $\Pi_2^X$ are idempotents) implies that 
            \[  \bigl((\Delta_G\circ \Pi_2^X)^{\circ N}\bigr){}_\ast=   (\Delta_G\circ \Pi_2^X){}_\ast\colon\ \ \ A^i(X)\ \to\ A^i(X)\ ,\]
            and so we find that
            \[    \bigl(\Delta_G\circ \Pi_2^X\bigr){}_\ast= (\Delta_G)_\ast=  0\colon\ \ \ A^2_{(2)}(X)\ \to\ A^2(X)\ .\]     
        This proves equality (\ref{22}).    
        
   The argument for $A^4_{(2)}(X)$ is similar: the correspondence $\Gamma$ being homologically trivial, its transpose
   \[ {}^t \Gamma= \Pi_6^X\circ \Delta_G -\gamma^{\prime\prime}\ \ \ \in A^4(X\times X) \]
   is also homologically trivial (where $\gamma^{\prime\prime}$ is supported on $D\times D$). Using the nilpotence theorem and proposition \ref{compat}, this implies (just as above) that
   \begin{equation}\label{42}   ( \Pi_6^X\circ \Delta_G){}_\ast=    \bigl(\Delta_G\circ \Pi_6^X\bigr){}_\ast= (\Delta_G)_\ast=  0\colon\ \ \ A^4_{(2)}(X)\ \to\ A^4(X)\ .\end{equation}
        
  In the final step of the proof, it remains to consider the action on $A^4_{(4)}(X)$. 
  Ideally, one would like to use Vial's projector $\pi^X_{4,0}$ of \cite[Theorems 1 and 2]{V4} (mentioned in the proof of theorem \ref{pi20}). Unfortunately, this approach runs into problems. 
  
  (The problem is that it seems impossible to prove that 
    \[ \Delta_G\circ \pi^X_{4,0}=0\ \ \ \hbox{in}\ H^8(X\times X)\ ,\] 
    short of knowing that (1) $H^4(X)\cap F^1 =N^1 H^4(X)$, and (2)  $N^1 H^4(X)=\wt{N}^1 H^4(X)$, where $N^\ast$ is the usual coniveau filtration and $\wt{N}^\ast$ is Vial's niveau filtration. Both (1) and (2) seem difficult.)
  
 We therefore proceed somewhat differently: to establish the statement for $A^4_{(4)}(X)$, we use the following proposition:
   
  \begin{proposition}\label{44} Let $Y\subset\PP^5(\C)$ be a smooth cubic fourfold, and let $X=F(Y)$ be the Fano variety of lines in $Y$. Let $G\subset\aut(X)$ be a group of non--symplectic automorphisms.
   One has
    \[ \Delta_G\circ \Pi_4^X - R =0\ \ \ \hbox{in}\ H^8(X\times X)\ ,\]
    where $R\in A^4(X\times X)$ is a correspondence with the property that
      \[  R_\ast=0\colon\ \ \ A^4(X)\ \to\ A^4(X)\ .\]
    \end{proposition}
    
    \begin{proof} This is \cite[Proposition 3.8]{nonsymp3}.
     \end{proof}
    
    Obviously, proposition \ref{44} clinches the proof: using the nilpotence theorem, one sees that there exists $N\in\NN$ such that
    \[ \bigl(   \Delta_G\circ \Pi_4^X + R \bigr){}^{\circ N}=0\ \ \ \hbox{in}\ A^4(X\times X)\ .\]
   Developing, and applying the result to $A^4(X)$, one finds that
    \[       \bigl(  ( \Delta_G\circ \Pi_4^X)^{\circ N}\bigr){}_\ast=0\colon\ \ \ A^4(X)\ \to\ A^4(X)\ .\]
   Proposition \ref{compat} (combined with the fact that $\Delta_G$ and $\Pi_4^X$ are idempotents) implies that 
            \[  \bigl((\Delta_G\circ \Pi_4^X)^{\circ N}\bigr){}_\ast=        (\Delta_G\circ \Pi_4^X){}_\ast\colon\ \ \ A^i(X)\ \to\ A^i(X)\ \ \ \forall i\not=2\ .\]
         Therefore, we conclude that
            \[    \bigl(\Delta_G\circ \Pi_4^X\bigr){}_\ast= (\Delta_G)_\ast=  0\colon\ \ \ A^4_{(4)}(X)\ \to\ A^4(X)\ ,\] 
            and we are done.  

   \end{proof}

 \begin{remark} The family of cubic fourfolds of theorem \ref{main0} also occurs in \cite{ACT} and \cite{LS} (where it is used as a tool in understanding the period map for cubic threefolds), as well as in \cite{vGI} (where it is shown the Kuga--Satake construction for these cubic fourfolds is algebraic).
  \end{remark}

\section{The two remaining families}

In this section, we consider the two remaining families of cubic fourfolds with an order $3$ automorphism that acts non--symplectically on the Fano variety. The main result here is theorem \ref{main}, stating that part of conjecture \ref{conjhk} is true for these two families. In addition, in theorem \ref{main2} we exhibit 
two subfamilies (of the two families of theorem \ref{main}) for which conjecture \ref{conjhk} is completely verified.

\subsection{Main result}

\begin{theorem}\label{main} Let $(Y,\sigma_Y)$ be one of the following:
  
  \noindent
  (a)
  $Y\subset\PP^5(\C)$ is a smooth cubic fourfold defined by an equation
    \[ f(X_0,X_1,X_2)+ g(X_3,X_4)+ (X_5)^3 + X_5\bigl( X_3\ell_1(X_0,X_1,X_2)+X_4\ell_2(X_0,X_1,X_2)\bigr)=0\ ,\]
    where $f,g$ are homogeneous polynomials of degree $3$ and $\ell_1,\ell_2$ are linear forms. $\sigma_Y\in\aut(Y)$ is the order $3$ automorphism induced by
    \[  \begin{split} \PP^5(\C)\ &\to\ \PP^5(\C)\ ,\\
                         [X_0:\ldots:X_5]\ &\mapsto\ [X_0:X_1:X_2:\nu X_3: \nu X_4:\nu^2 X_5]\\
                         \end{split}\]
    (where $\nu$ is a $3$rd root of unity). 
    
  \noindent
  (b)
    $Y\subset\PP^5(\C)$ is a smooth cubic fourfold defined by an equation
    \[ \begin{split} X_2 f(X_0,X_1)+ X_3 g(X_0,X_1)+ (X_4)^2\ell_1(X_0,X_1) + X_4X_5\ell_2(X_0,X_1)+&\\ (X_5)^2\ell_3(X_0,X_1)+ X_4 h(X_2,X_3) +X_5 k(X_2,X_3)&=0\ ,\\
    \end{split}\]
    where $f,g,h,k$ are homogeneous polynomials of degree $2$ and $\ell_1,\ell_2$ are linear forms. $\sigma_Y\in\aut(Y)$ is the order $3$ automorphism induced by
    \[  \begin{split} \PP^5(\C)\ &\to\ \PP^5(\C)\ ,\\
                         [X_0:\ldots:X_5]\ &\mapsto\ [X_0:X_1:\nu X_2:\nu X_3: \nu^2 X_4:\nu^2 X_5]\\
                         \end{split}\]
    (where $\nu$ is a $3$rd root of unity). 
    
  Let $X=F(Y)$ be the Fano variety of lines in $Y$, and $\sigma\in\aut(X)$ the non--symplectic automorphism induced by $\sigma_Y$.  
    Then
    \[ \begin{split}  (\ide +\sigma+\sigma^2)_\ast \ A^2_{(2)}(X)&=0\ ,\\
                               (\ide +\sigma+\sigma^2)_\ast \ A^4_{(2)}(X)&=0\ .\\
                               \end{split}        \]
  \end{theorem}

 \begin{proof}  
 (NB: The two families of Fano varieties of theorem \ref{main} are described in \cite[Examples 6.6 and 6.7]{BCS}, where it is established that the automorphism 
 $\sigma$ is non--symplectic in both cases.)
                             
   In a first step of the proof, we show that the automorphism $\sigma$ respects the Fourier decomposition of the Chow ring:
 
 \begin{proposition}\label{compatagain} Let $X$ and $\sigma$ be as in theorem \ref{main}. Let $A^\ast_{(\ast)}(X)$ be the Fourier decomposition (theorem \ref{sv}).
 Then
   \[ \sigma_\ast \, A^i_{(j)}(X)\ \subset\ A^i_{(j)}(X)\ \ \ \forall (i,j)\ .\]
   Equivalently, if $\{\Pi_j^X\}$ is a CK decomposition as in theorem \ref{sv}, then we have
    \[ \sigma_\ast (\Pi_j^X)_\ast   = (\Pi_j^X)_\ast \sigma_\ast  (\Pi_j^X)_\ast\colon\ \ \ A^i(X)\ \to\ A^i(X)\ \ \ \forall (i,j)\ .\]  
     \end{proposition}
  
  \begin{proof} The proof is exactly the same as that of proposition \ref{compat}.  
          \end{proof}

  In the next step, we look at the action of $\sigma$ on $H^2(X)$ (where $(X,\sigma)$ is the Fano variety with the order $3$ automorphism as in theorem \ref{main}(a) or (b)). Let us use the short--hand 
    \[ \Delta_G:={1\over 3}\bigl(\Delta_X +\Gamma_\sigma+\Gamma_{\sigma^2}\bigr)\ \ \ \in A^4(X\times X)\ .\]
    The anti--symplectic condition translates into the fact that
    \[ (\Delta_G)_\ast=0\colon\ \ \ H^{2,0}(X)\ \to\ H^{2,0}(X)\ .\]
   Since the kernel of $(\Delta_G)_\ast$ is a sub--Hodge structure, and $H^2_{tr}(X)\subset H^2(X)$ is the smallest sub--Hodge structure containing $H^{2,0}(X)$, it follows that also
     \[ (\Delta_G)_\ast=0\colon\ \ \ H^{2}_{tr}(X)\ \to\ H^{2}_{tr}(X)\ .\]
    This means that
    \[ \Delta_G\circ \pi_{2,0}^X=0\ \ \ \hbox{in}\ H^8(X\times X) \]
    (where $\pi_{2,0}^X$ is as in theorem \ref{pi20}), and so
    \[ \Delta_G\circ \pi_2^X = \gamma \ \ \ \hbox{in}\ H^8(X\times X) \ ,\]
    where $\gamma$ is a cycle supported on $C\times D\subset X\times X$, where $C$ is a curve and $D$ is a divisor in $X$. In particular, since the projector 
    $\Pi_2^X$ of theorem \ref{sv} is equal to $\pi_2^X$ in cohomology, we get a homological equivalence
    \begin{equation}\label{onex}    \Delta_G\circ \Pi_2^X = \gamma \ \ \ \hbox{in}\ H^8(X\times X) \ ,\end{equation}
    where $\gamma$ is supported on $C\times D$.
    
    Now we can start to play the ``let's spread things out'' game of \cite{V0}, \cite{V1}. 
   We define
   \[ \pi\colon\ \ \YY\ \to\ B \]
  to be the family of all smooth cubic fourfolds given by an equation as in theorem \ref{main}. (That is, we let $G\subset\aut(\PP^5)$ be the order $3$ group generated by the automorphism of $\PP^5$ as in (a) or (b), and we define
    \[ B\ \subset\ \Bigl(\PP H^0\bigl(\PP^5,\OO_{\PP^5}(3)\bigr)\Bigr)^G \]
    as the Zariski open subset parametrizing smooth $G$--invariant cubics. Note that $\YY\to B$ is thus either the family of (a), or the family of (b); we treat the two families simultaneously.)  
    
 We will write $
    Y_b:=\pi^{-1}(b)$ for the fibre over $b\in B$. We also define the universal family of associated Fano varieties of lines:
    \[ \XX :=  \bigl\{ (\ell,b)\ \vert\  \ell\subset Y_b\bigr\}\ \ \ Gr\times B  \ ;   \]
   the fibre of $\XX$ over $b\in B$ is the Fano variety $X_b$ of lines in $Y_b$. Since $\XX$ is invariant under $\sigma_{Gr}\times\ide_B$, there is an automorphism
     \[ \sigma\colon\ \ \ \XX\ \to\ \XX\ .\]
     We will write $\sigma_b\colon X_b\to X_b$ for the restriction to a fibre. Let us define a relative correspondence
     \[ \Delta_G^\XX:={1\over 3}\, \bigl( \Delta_\XX +\Gamma_\sigma + \Gamma_{\sigma^2}\bigr)\ \ \ \in A^4(\XX\times_B \XX)\ ,\]
     and consider the composition 
      \[ \Delta_G^\XX\circ \Pi_2^\XX\ \ \ \in A^4(\XX\times_B \XX)\ .\]  
 (For composition of relative correspondences in the setting of smooth quasi--projective families that are smooth over a base $B$, cf.   
   \cite{CH}, \cite{GHM}, \cite{NS}, \cite{DM}, \cite[8.1.2]{MNP}.) 
      
      In view of (\ref{onex}) above, this correspondence has the following property: for every $b\in B$, there exist a curve $C_b$ and a divisor $D_b\subset X_b$, and a cycle $\gamma_b$ supported on $C_b\times D_b$, such that
      \[ \bigl(\Delta_G^\XX\circ \Pi_2^\XX\bigr)\vert_{X_b\times X_b}=\gamma_b
          \ \ \ \in H^8(\XX\times_B \XX)\ .\]  
         
     Applying Voisin's Hilbert schemes argument \cite[Proposition 3.7]{V0}, we can find a ``completely decomposed'' relative correspondence $\gamma\in A^4(\XX\times_B \XX)$
    such that
    \begin{equation}\label{vanish}  \bigl(  \Delta_G^\XX\circ \Pi_2^\XX      -\gamma\bigr)\vert_{X_b\times X_b} =0\ \ \ \hbox{in}\ H^8(  X_b\times X_b)
       \ \ \ \forall b\in B\ .\end{equation}
    (By ``completely decomposed'' we mean the following: there exist subvarieties $\CC, \DD\subset \XX$ of codimension $3$ resp. $1$, such that the cycle $\gamma$ is
    supported on $\CC\times_B \DD\subset \XX\times_B \XX$.)
    
    Now, we would like to apply proposition \ref{voisin2} (which acts as a magic wand, transforming a homological equivalence into a rational equivalence). For 
    this reason, it is more convenient to move things to the family $\YY$. That is, we define a relative correspondence
      \[  \Gamma:=     {}^t \Psi\circ\Bigl(   \Delta_G^\XX\circ \Pi_2^\XX      -\gamma\Bigr)\circ {}^t \PPP\ \ \ \in A^4(\YY\times_B \YY)       \ ,\]
      where $\Psi$ and $\PPP$ are as in corollary \ref{decomp}. (NB: in the present set--up, the base $B$ is different from corollary \ref{decomp}; our present $B$ is actually $B^\sigma$. Let $B_{univ}$ be the base parametrizing all smooth cubics (as in corollary \ref{decomp}), and let $B\subset B_{univ}$ be the base parametrizing $\sigma$--invariant cubics (as in the present proof). Then the present $\Psi$ and $\PPP$ are obtained from $\Psi_{univ}, \PPP_{univ}$ (of corollary \ref{decomp}) by pullback
      \[   \Psi:= \tau^\ast(\Psi_{univ})\ ,\ \ \ {}^t\PPP:=\tau^\ast({}^t \PPP_{univ})\ ,\]
      where $\tau\colon \YY\times_B \XX\to \YY_{univ}\times_{B_{univ}} \XX_{univ}$ is the base change.)

       Property (\ref{vanish}) implies there is a fibrewise homological vanishing
      \[ \Gamma\vert_{Y_b\times Y_b}=0\ \ \ \hbox{in}\ H^8(Y_b\times Y_b)\ \ \ \forall b\in B\ .\]
      Applying proposition \ref{voisin2} (which is possible thanks to lemma \ref{OK} below), we find there exists $\delta\in A^4(\PP^5\times \PP^5)$ such that
      \begin{equation}\label{this}  \Gamma\vert_{Y_b\times Y_b}=\delta\vert_{Y_b\times Y_b}\ \ \ \hbox{in}\ A^4(Y_b\times Y_b)\ \ \ \forall b\in B\ .\end{equation}
      Since $A^\ast_{hom}(\PP^5)=0$, the restriction $\delta\vert_{Y_b\times Y_b}$ acts trivially on $A^\ast_{hom}(Y_b)$, and thus
      \[       \bigl(  \Gamma\vert_{Y_b\times Y_b}\bigr){}_\ast=0\colon\ \ \ A^\ast_{hom}(Y_b)\ \to\ A^\ast_{hom}(Y_b)\ \ \ \forall b\in B\ .\]
      Composing on both sides, it follows that also
      \[  \bigl( ( {}^t \PPP\circ \Gamma\circ {}^t \Psi)\vert_{X_b\times X_b}\bigr){}_\ast=0\colon\ \ \     A^\ast_{hom}(X_b)\ \to\ A^\ast_{hom}(X_b)\ \ \ \forall b\in B\ \]
      (where $\PPP$ and $\Psi$ are as in corollary \ref{decomp}).
      By definition of $\Gamma$, this means that
     \[  \bigl( ( {}^t \PPP\circ {}^t \Psi\circ (\Delta_G^\XX\circ \Pi_2^\XX-\gamma)\circ {}^t \PPP  \circ {}^t \Psi)\vert_{X_b\times X_b}\bigr){}_\ast=0\colon\ \ \    
          A^\ast_{hom}(X_b)\ \to\ A^\ast_{hom}(X_b)\ \ \ \forall b\in B\ .\]
          In the light of corollary \ref{decomp}, this boils down to
        \[ \bigl( \Pi_2^{X_b}\circ (\Delta_G^\XX\circ \Pi_2^\XX-\gamma)\vert_{X_b\times X_b}\circ \Pi_2^{X_b}\bigr){}_\ast  =0\colon\ \ \    
          A^2_{hom}(X_b)\ \to\ A^2_{hom}(X_b)\ \ \ \forall b\in B\ .\]
       But $(\Pi_2^{X_b})_\ast A^2_{hom}(X_b)=A^2_{(2)}(X_b)$ and so
        \[ \bigl( \Pi_2^{X_b}\circ (\Delta_G^\XX\circ \Pi_2^\XX-\gamma)\vert_{X_b\times X_b}\bigr){}_\ast  =0\colon\ \ \    
          A^2_{(2)}(X_b)\ \to\ A^2_{}(X_b)\ \ \ \forall b\in B\ .\]  
     We can rewrite this as
      \[       \bigl( \Pi_2^{X_b}\circ \Delta_G^{X_b}\circ \Pi_2^{X_b}  - \Pi_2^{X_b}\circ (\gamma\vert_{X_b\times X_b})\bigr){}_\ast =0\colon\ \ \    
          A^2_{(2)}(X_b)\ \to\ A^2_{}(X_b)\ \ \ \forall b\in B\ .\]  
          But for general $b\in B$, the restriction $\gamma\vert_{X_b\times X_b}$ will be supported on (curve)$\times$(divisor), and as such will act trivially on $A^2(X_b)$. It follows that
          \[  \bigl( \Pi_2^{X_b}\circ \Delta_G^{X_b} \bigr){}_\ast=0\colon\ \ \ A^2_{(2)}(X_b)\ \to\ A^2_{}(X_b)\ \ \ \hbox{for\ general\ } b\in B\ .\]  
          In view of proposition \ref{compat}, this implies that
         \[ ( \Delta_G^{X_b} ){}_\ast=0\colon\ \ \ A^2_{(2)}(X_b)\ \to\ A^2_{}(X_b)\ \ \ \hbox{for\ general\ } b\in B\ ,\]
          i.e. we have proven the first statement of theorem \ref{main} for general $b\in B$. To extend this to {\em all\/} $b\in B$, we observe that the above construction can be made locally around a given $b_0\in B$: the point is that in applying \cite[Proposition 3.7]{V0}, one can obtain a cycle $\gamma_{}$ supported on $\CC\times_B \DD$ such that $\CC$ and $\DD$ are in general position with respect to $X_b$. The above argument then proves the statement for $X_{b_0}$.
          
        The second statement of theorem \ref{main} (i.e., the action on $A^4_{(2)}(X)$) can be recovered from the above argument. Indeed, taking the transpose on both sides of equation (\ref{this}), one obtains that for all $b\in B$, there is equality
     \[ {}^t   ( \Gamma\vert_{Y_b\times Y_b})={}^t  (\delta\vert_{Y_b\times Y_b})\ \ \ \hbox{in}\ A^4(Y_b\times Y_b)\ .\]   
     As the right--hand side acts trivially on $A^\ast_{hom}(Y_b)$, this implies that
     \[ \bigl(  {}^t   ( \Gamma\vert_{Y_b\times Y_b})\bigr){}_\ast =0\colon\ \ \ A^\ast_{hom}(Y_b)\ \to\ A^\ast_{hom}(Y_b)\ \ \ \forall b\in B\ .\]
     By definition of $\Gamma$ (combined with the fact that $\Pi_6^{X_b}={}^t \Pi_2^{X_b}$), this means that
      \[  \bigl( ( {} \Psi\circ {} \PPP\circ (\Pi_6^\XX\circ\Delta_G^\XX -{}^t \gamma)\circ {} \Psi  \circ {} \PPP)\vert_{X_b\times X_b}\bigr){}_\ast=0\colon\ \ \    
          A^\ast_{hom}(X_b)\ \to\ A^\ast_{hom}(X_b)\ \ \ \forall b\in B\ .\]
     Using corollary \ref{decomp}, this simplifies to
     \[       \bigl( \Pi_6^{X_b}\circ (\Pi_6^\XX\circ\Delta_G^\XX -{}^t \gamma)\vert_{X_b\times X_b}\circ \Pi_6^{X_b}  \bigr){}_\ast=0\colon\ \ \    
          A^4_{hom}(X_b)\ \to\ A^4_{hom}(X_b)\ \ \ \forall b\in B\ ,\]
       i.e.
       \[    \bigl( \Pi_6^{X_b}\circ (\Delta_G^\XX -{}^t \gamma)\vert_{X_b\times X_b} \bigr){}_\ast=0\colon\ \ \    
          A^4_{(4)}(X_b)\ \to\ A^4_{}(X_b)\ \ \ \forall b\in B\ .\]   
        This can be rewritten as
      \[       \bigl( \Pi_6^{X_b}\circ \Delta_G^{X_b}  - \Pi_6^{X_b}\circ ({}^t (\gamma\vert_{X_b\times X_b}))\bigr){}_\ast =0\colon\ \ \    
          A^4_{(2)}(X_b)\ \to\ A^4_{}(X_b)\ \ \ \forall b\in B\ .\]   
          For general $b\in B$, the transpose of the restriction ${}^t (\gamma\vert_{X_b\times X_b})$ is supported on (divisor)$\times$(curve), and hence will act trivially on $A^4(X_b)$. It follows that
          \[  \bigl( \Pi_6^{X_b}\circ \Delta_G^{X_b}\bigr){}_\ast =0\colon\ \ \ A^4_{(2)}(X_b)\ \to\ A^4(X_b)\ \ \      \hbox{for\ general\ } b\in B\ .\]  
       As $\sigma$ preserves the Fourier decomposition (proposition \ref{compat}), this implies that
        \[ (   \Delta_G^{X_b}){}_\ast =0\colon\ \ \ A^4_{(2)}(X_b)\ \to\ A^4(X_b)\ \ \      \hbox{for\ general\ } b\in B\ .\]    
        As above, this can be extended to {\em all\/} $b\in B$.
        
      Since we have used proposition \ref{voisin2}, we need to check the hypotheses of this proposition are satisfied. This is taken care of by the following lemma:
        
\begin{lemma}\label{OK} Let $\YY\to B$ be one of the families of cubics as in (a) or (b) of theorem \ref{main}. Then $\YY\to B$ verifies the hypotheses of proposition \ref{voisin2} (with $M=\PP^5(\C)$, $L=\OO_{\PP^5}(3)$ and $G=<\sigma_\PP>$, where $\sigma_\PP\in\aut(\PP^5)$ is as in theorem \ref{main}(a) resp. (b)).
    \end{lemma}
        
    \begin{proof} First, let us check that hypothesis (\rom1) of proposition \ref{voisin2} is met with in case (a). To this end, we consider the tower of morphisms
   \[ p\colon\ \ \PP^5\ \xrightarrow{p_1}\ P^\prime:= \PP^5/G\ \xrightarrow{p_2}\ P:=\PP(1^3,3^3)\ ,\]
   where $\PP(1^3,3^3)=\PP^5/(\ZZ/3\ZZ\times \ZZ/3\ZZ\times \ZZ/3\ZZ)$ denotes a weighted projective space. Let us write $\sigma_3, \sigma_4, \sigma_5$ for the order $3$ automorphisms of $\PP^5$ 
   \[  \begin{split}  &\sigma_3 [x_0:\ldots:x_5] = [x_0:x_1:x_2:\nu x_3: x_4:x_5]\ ,\\
                               &\sigma_4 [x_0:\ldots:x_5] = [x_0:x_1:x_2:x_3:\nu x_4:x_5]\ ,\\
                            &\sigma_5 [x_0:\ldots:x_5] = [x_0:x_1:x_2:x_3:x_4:\nu x_5]\ ,\\
                        \end{split} \] 
                        where $\nu$ is again a primitive $3$rd root of unity.   
   (We note that $\sigma_\PP=\sigma_3\circ\sigma_4\circ\sigma_5$, and the weighted projective space $P$ is $\PP^5/ <\sigma_3,\sigma_4,\sigma_5>$.)
   
   The sections in $\bigl(\PP H^0\bigl(\PP^5,\OO_{\PP^5}(3)\bigr)\bigr)^G$ are in bijection with $ \PP H^0\bigl(P^\prime,\OO_{P^\prime}(3)\bigr) $, and there is an inclusion
   \[  \Bigl(\PP H^0\bigl(\PP^5,\OO_{\PP^5}(3)\bigr)\Bigr)^G \ \supset\ p^\ast \PP H^0\bigl( P,\OO_P(3)\bigr)\ .\]  
   
   Let us now assume $x,y\in\PP^5$ are two points such that
   \[ (x,y)\not\in \Delta_{\PP^5}\cup \bigcup_{r_3,r_4,r_5\in\{0,1,2\}} \Gamma_{(\sigma_3)^{r_3}\circ (\sigma_4)^{r_4}\circ (\sigma_5)^{r_5}}\ .\]
   Then 
   \[ p(x)\not=p(y)\ \ \ \hbox{in}\ P\ ,\]
   and so (using lemma \ref{delorme} below) there exists $\sigma\in\PP H^0\bigl(P^\prime,\OO_{P^\prime}(3)\bigr)$ containing $p(x)$ but not $p(y)$. The pullback $p^\ast(\sigma)$ contains $x$ but not $y$, and so these points $(x,y)$ impose $2$ independent conditions on the parameter space $ \bigl(\PP H^0\bigl(\PP^5,\OO_{\PP^5}(3)\bigr)\bigr)^G$.
   
   It only remains to check that a generic element 
    \[ (x,y)\ \ \ \in\Gamma_{(\sigma_3)^{r_3}\circ (\sigma_4)^{r_4}\circ (\sigma_5)^{r_5}}  \] 
    also imposes $2$ independent conditions, for any $(r_3,r_4,r_5)\not\in\{(0,0,0),(1,1,1),(2,2,2)\}$. Let us assume $(x,y)$ is generic on $\Gamma_{\sigma_3}$ (the argument for the other $r_i$ are similar). Let us write $x=[a_0:a_1:\ldots:a_5]$. By genericity, we may assume all $a_i$ are $\not=0$ (intersections of $\Gamma_{\sigma_3}$ with a coordinate hyperplane have codimension $>n+1$ and so need not be considered for hypothesis (\rom1) of proposition \ref{voisin2}). We can thus write
   \[ x=[1:a_1:a_2:a_3:a_4:a_5]\ ,\ \ \ y= [1:a_1:a_2:\nu a_3:a_4:a_5] \ ,\ \ \ a_i\not=0 \ .\] 
   The cubic 
   \[  a_4 (x_0)^2x_3 - a_3(x_0)^2x_4=0 \]
   is $G$--invariant and contains $x$ while avoiding $y$, and so the element $(x,y)$ again imposes $2$ independent conditions. This proves hypothesis (\rom1) is satisfied for case (a).
   
  The proof of hypothesis (\rom1) for case (b) is similar, using the weighted projective space $\PP(1^2,3^4)$.
  
   To establish hypothesis (\rom2) of proposition \ref{voisin2}, we proceed by contradiction. Let us suppose hypothesis (\rom2) is not met with, i.e. there exists a smooth cubic $Y=Y_b$ as in theorem \ref{main}(b) or (c), and a non--trivial relation
  \[  a\,\Delta_{Y} +b\, \Gamma_{\sigma} + c\, \Gamma_{\sigma^2}+ \delta =0\ \ \ \hbox{in}\ H^8(Y\times Y)\ ,\]
  where $a,b,c\in \QQ$ and $\delta\in\ima\bigl( A^4(\PP^5\times\PP^5)\to A^4(Y\times Y)\bigr)$.
  Looking at the action on a generator $\omega$ of $H^{3,1}(Y)$ (and using that $\delta_\ast(\omega)=0$ because $H^{p,q}(\PP^5)=0$ for $p\not=q$), we find a relation
   \[ a+ \nu b + \nu^2 c=0\ .\]
   Next, looking at the action on $(H^4(Y)_{prim})^\sigma$ (which is non--zero, for there is an equivariant isomorphism $H^4(Y)\cong H^2(X)$, and proposition \ref{chiara} below ensures that $\dim H^2(X)^\sigma>1$), we find a relation
   \[ a+b+c=0\ .\]
Combining these two relations, we find that the only solutions are $a=\nu c$, $b=\nu^2 c$. It follows that there are no rational solutions, and so hypothesis (\rom2) is satisfied.   
    
  \begin{lemma}\label{delorme} Let $P$ be either $\PP(1^3,3^3)$ or $\PP(1^2,3^4)$. Let $r,s\in P$ and $r\not=s$. Then there exists $\sigma\in\PP H^0\bigl(P,\OO_P(3)\bigr)$ containing $r$ but avoiding 
  $s$.
  \end{lemma}
  
  \begin{proof} It follows from Delorme's work \cite[Proposition 2.3(\rom3)]{Del} that the locally free sheaf $\OO_P(3)$ is very ample. This means there exists 
  $\sigma$ as required.
    \end{proof}
  
\begin{proposition}[Boissi\`ere--Camere--Sarti \cite{BCS}]\label{chiara} Let $(X,\sigma)$ be a Fano variety $X=F(Y)$ with a polarized order $3$ automorphism $\sigma\in\aut(X)$ as in theorem \ref{main}(a) or (b). Let $T\subset NS(X)$ denote the $\sigma$--invariant part of the N\'eron--Severi group of $X$.

The dimension of $T$ equals $7 $ in case (a), and equals $9$ in case (b).
\end{proposition}  

\begin{proof} The lattice $T$ is completely determined in \cite[Examples 6.6 and 6.7]{BCS}.
  \end{proof}
  
  \end{proof}

        \end{proof}

 \begin{remark}\label{problem} To prove the full conjecture \ref{conjhk} for the two families of theorem \ref{main}, it remains to prove that
   \[ (\Delta_G)_\ast A^4_{(4)}(X)\stackrel{??}{=}0\ .\]
  The above argument is not strong enough to prove this, for the following reason: at one point in the argument, we moved from the family $\XX$ (of Fano varieties) to the family $\YY$ (of cubics), using the correspondences $\Psi$ and $\PPP$ of corollary \ref{decomp}. This worked because (thanks to corollary \ref{decomp}) this moving back and forth did {\em not\/} affect $A^4_{(2)}(X)$ and $A^2_{(2)}(X)$. However, this cannot work for $A^4_{(4)}(X)$.
  
  (One might expect, in light of \cite{fano}, that $A^4_{(4)}(X)$ is related to $A^6(Y^{[2]})$. If so, one could move from $A^\ast(\XX\times_B \XX)$ to $A^\ast( \YY^{4/B})$. Then, one would need a version of proposition \ref{voisin2} for the $4$th fibre product $\YY^{4/B}$; this seems difficult !) 
   \end{remark}

 \subsection{Two subfamilies}
 
  In this subsection, we restrict to two special subfamilies (of the families of theorem \ref{main}), in order for Kimura's notion of {\em finite--dimensional motive\/} \cite{Kim} to apply.

  \begin{theorem}\label{main2} Let $(Y,\sigma_Y)$ be one of the following:
  
  \noindent
  (a)
  $Y\subset\PP^5(\C)$ is a smooth cubic fourfold defined by an equation
    \[ f(X_0,X_1,X_2)+ g(X_3,X_4)+ (X_5)^3 =0\ ,\]
   and $\sigma_Y\in\aut(Y)$ is as in theorem \ref{main}(a).
   
     \noindent
  (b)
    $Y\subset\PP^5(\C)$ is a smooth cubic fourfold defined by an equation
    \[  f(X_0,X_1)+g(X_2,X_3)+h(X_4,X_5)=0\ ,\]
   and $\sigma_Y\in\aut(Y)$ is as in theorem \ref{main}(b).
    
  Let $X=F(Y)$ be the Fano variety of lines in $Y$, and let $\sigma\in\aut(X)$ be the non--symplectic automorphism induced by $\sigma_Y$.  
    Then
    \[ (\ide +\sigma+\sigma^2)_\ast \ A^4_{hom}(X)=0\ .\]
\end{theorem}

  \begin{proof} Since we have
    \[ A^4_{hom}(X)=A^4_{(4)}(X)\oplus A^4_{(2)}(X)\ ,\]
    and we already know (theorem \ref{main}) that
     \[ (\ide +\sigma+\sigma^2)_\ast \ A^4_{(2)}(X)=0\ ,\]  
   it suffices to prove that also
      \begin{equation}\label{sub4} (\ide +\sigma+\sigma^2)_\ast \ A^4_{(4)}(X)=0\ .\end{equation}    
       
  We will exploit the following fact:
  
  \begin{lemma}\label{findim} Let $X$ be as in theorem \ref{main2}. Then $X$ has finite--dimensional motive.
  \end{lemma}
  
  \begin{proof} One first proves the cubics $Y$ as in theorem \ref{main2} have finite--dimensional motive. This follows from Shioda's inductive structure \cite{Sh}, \cite[Remark 1.10]{KS}, combined with the fact that cubic curves, cubic surfaces and cubic threefolds have finite--dimensional motive. The result for $X=F(Y)$ now follows from \cite{fano}.
    \end{proof}
  
  Now, we go on to prove (\ref{sub4}). Proposition \ref{44} applies, and so we have
   \[ \Delta_G\circ \Pi_4^X - R =0\ \ \ \hbox{in}\ H^8(X\times X)\ ,\]
    where $R\in A^4(X\times X)$ is a correspondence acting trivially on $A^4(X)$. 
 We conclude just as at the end of the proof of theorem \ref{main0}: using the nilpotence theorem (theorem \ref{nilp}), one finds $N\in\NN$ such that
    \[ \bigl(   \Delta_G\circ \Pi_4^X + R \bigr){}^{\circ N}=0\ \ \ \hbox{in}\ A^4(X\times X)\ .\]
   Developing, and applying the result to $A^4(X)$, it follows that
    \[       \bigl(  ( \Delta_G\circ \Pi_4^X)^{\circ N}\bigr){}_\ast=0\colon\ \ \ A^4(X)\ \to\ A^4(X)\ .\]
    Proposition \ref{compat} (combined with the fact that $\Delta_G$ and $\Pi_4^X$ are idempotents) implies that 
            \[  \bigl((\Delta_G\circ \Pi_4^X)^{\circ N}\bigr){}_\ast=        (\Delta_G\circ \Pi_4^X){}_\ast\colon\ \ \ A^i(X)\ \to\ A^i(X)\ \ \ \forall i\ .\]
         Therefore, we may conclude that
            \[    \bigl(\Delta_G\circ \Pi_4^X\bigr){}_\ast= (\Delta_G)_\ast=  0\colon\ \ \ A^4_{(4)}(X)\ \to\ A^4(X)\ .\]   
    \end{proof}

\section{Some corollaries}

\subsection{A succinct restatement}
   
  \begin{corollary}\label{done}  Let $X$ be the Fano variety of lines of a smooth cubic fourfold. Let $\sigma\in\aut(X)$ be a polarized, prime order automorphism that is non--symplectic. Then 
  \[   (\Delta_G)_\ast A^i_{(2)}(X)=0\ \ \ \hbox{for}\ i\in\{2,4\}\ .\]
   \end{corollary}
  
  \begin{proof} The point is that smooth cubic fourfolds with automorphisms of prime order have been classified. If the induced automorphism on the Fano variety $X$ is non--symplectic, the only possibilities are that $\sigma$ is of order $2$ or $3$ (\cite{GAL}, combined with the corrigendum in \cite[Remark 6.3]{BCS} to rule out the case of order $5$). 
  
  For $\sigma$ of order $2$, there are two families of cubic fourfolds, and these have been treated in \cite{inv}, \cite{inv2}. For $\sigma$ of order $3$, there are $4$ families \cite[Examples 6.4, 6.5, 6.6 and 6.7]{BCS}; the first is treated in \cite{nonsymp3}, the others in theorems \ref{main0} and \ref{main}.
  \end{proof}

  \subsection{Bloch conjecture}

   \begin{corollary}\label{triv} Let $X$ and $\sigma$ be as in theorem \ref{main0} or theorem \ref{main2}. Let $Z:=X/<\sigma>$ be the quotient. Then
  \[ A^4(Z)\cong \QQ\ .\]
 \end{corollary}
 
 \begin{proof} This readily follows from the natural isomorphism $A^4(Z)\cong A^4(X)^\sigma$, cf. \cite[Proof of Corollary 4.1]{nonsymp3}.
 \end{proof}

\subsection{Generalized Hodge conjecture}

  \begin{corollary}\label{ghc} Let $X$ and $\sigma$ be as in theorem \ref{main0} or theorem \ref{main2}. Then the invariant part of cohomology
   \[ H^4(X)^\sigma  \ \subset \ H^4(X) \]
   is supported on a divisor.
   \end{corollary}
   
   \begin{proof} This is an application of the Bloch--Srinivas argument \cite{BS}, cf. \cite[Proof of Corollary 4.2]{nonsymp3}.
%
   \end{proof}

 \subsection{Intersection of cycles}  
   
%
   
\begin{corollary}\label{ring} Let $X$ and $\sigma$ be as in theorem \ref{main0} or theorem \ref{main}. Let $a\in A^3(X)$ be a $1$--cycle of the form
  \[ a=\displaystyle\sum_{i=1}^r b_i\cdot D_i\ \ \ \in A^3(X)\ ,\]
  where $b_i\in A^2(X)^\sigma$ and $D_i\in A^1(X)_{}$. Then $a$ is rationally trivial if and only if $a$ is homologically trivial.
\end{corollary}

\begin{proof} Since $\sigma$ respects the bigrading $A^\ast_{(\ast)}(X)$ (proposition \ref{compat}), theorem \ref{main0} or theorem \ref{main} implies that
  \[ A^2(X)^\sigma\ \subset\ A^2_{(0)}(X)\ .\]
  Since $A^2_{(0)}(X)\cdot A^1(X)_{}\subset A^3_{(0)}(X)$ (this is \cite[Proposition A.7]{FLV}, which improves upon
  \cite[Proposition 22.7]{SV}), it follows that
  \[ a\ \ \in A^3_{(0)}(X)\ .\]
  But $A^3_{(0)}(X)$ injects into cohomology under the cycle class map \cite{SV}.
  ( A quick way of proving this injectivity can be as follows: let $\FF$ be the Fourier transform of \cite{SV}. We have that $a\in A^3(X)$ is in $A^3_{(0)}(X)$ if and only if $\FF(a)\in A^1_{(0)}(X)=A^1(X)$ \cite[Theorem 2]{SV}. Suppose $a\in A^3_{(0)}(X)$ is homologically trivial. Then also $\FF(a)\in A^1(X)$ is homologically trivial, hence $\FF(a)=0$ in $A^1(X)$. But then, using \cite[Theorem 2.4]{SV}, we find that
      \[ {25\over 2} a = \FF\circ \FF(a)=0\ \ \ \hbox{in}\ A^3(X)\ .)\]
    \end{proof}
               
%
  
 \begin{corollary}\label{cheat} Let $X$ and $\sigma$ be as in theorem \ref{main0} or theorem \ref{main}. Let $\phi\colon X\dashrightarrow X$ be the degree $16$ rational map first defined in \cite{V21}. Let $a\in A^2(X)$ be a $2$--cycle of the form
   \[ a=\phi^\ast(b)-4b\ \ \ \in A^2(X)\ ,\]
   where $b\in A^2(X)$ is a sum of $\sigma$--invariant cycles and intersections of divisors. Then $a$ is rationally trivial if and only if $a$ is homologically trivial.
 \end{corollary} 
 
 \begin{proof} We know from theorem \ref{main0} or theorem \ref{main} (combined with the fact that $A^1(X)\cdot A^1(X)\subset A^2_{(0)}(X)$ \cite[Lemma 4.4]{SV}) that $b$ is in $A^2_{(0)}(X)$. 
Let $V^2_\lambda$ denote the eigenspace 
    \[ V^2_\lambda:=\{ c\in A^2(X)\ \vert\ \phi^\ast(c)=\lambda\cdot c\}\ .\] 
Shen--Vial have proven that there is a decomposition
  \[ A^2_{(0)}(X)=V^2_{31}\oplus V^2_{-14}\oplus V^2_{4}\ \]
    \cite[Theorem 21.9]{SV}. The ``troublesome part'' $A^2_{(0)}(X)\cap A^2_{hom}(X)$ is contained in $V^2_4$ \cite[Lemma 21.12]{SV}. This implies that
   \[ (\phi^\ast-4(\Delta_X)^\ast)A^2_{(0)}(X)=V^2_{31}\oplus V^2_{-14} \]
 injects into cohomology.  
       \end{proof}

 \begin{remark} Corollary \ref{cheat} is probably less than optimal: conjecturally, for any $b\in A^2(X)$ as in corollary \ref{cheat},
 one should have that $b$ is rationally trivial if and only if $b$ is homologically trivial. The problem, in proving such a statement along the lines of the present note, is that we cannot prove that
   \[ A^2_{(0)}(X)\cap A^2_{hom}(X)\stackrel{??}{=}0\ .\] 
  \end{remark}

\vskip1cm
\begin{nonumberingt} Thanks to all participants of the Strasbourg 2014/2015 ``groupe de travail'' based on the monograph \cite{Vo} for a very stimulating atmosphere.
Many thanks to Mrs. Yasuyo Ishitani, Director of the Special Program ``O-Uchi De O-Shigoto'', for an excellent working environment.
\end{nonumberingt}

\end{document}